\documentclass[12pt,a4paper,reqno]{amsart}
\usepackage{amsfonts,amsthm,amsmath,amssymb,
	amsbsy,euscript,
txfonts, tikz} 

\newcommand{\by}[1]{\textit{{#1}}}
\newcommand{\jour}[1]{\textit{{#1}}}
\newcommand{\vol}[1]{\textbf{{#1}}}
\newcommand{\book}[1]{\textrm{{#1}}}

\newcommand{\sfE}{\mathsf{E}}
\newcommand{\sfV}{\mathsf{V}}
\newcommand{\edge}{{\bullet\!\!\!\frac{\quad}{}\!\!\!\bullet}}
\newcommand{\Id}{{\mathrm d}}
\newcommand*\circled[1]{\tikz[baseline=(char.base)]{
   \node[shape=circle,draw,inner sep=1pt] (char) {#1};}}
\newcommand{\zstroke}{%
  \text{\tiny\ooalign{\hidewidth\raisebox{0.2ex}{--}\hidewidth\cr$Z$\cr}}%
}

\newtheorem{theor}{Theorem}

\newtheorem*{claimNN}{Claim}
\theoremstyle{definition}

\newtheorem{proposition}[theor]{Proposition}
\newtheorem{lemma}[theor]{Lemma}
\newtheorem{cor}[theor]{Corollary}
\newtheorem{define}{Definition}

\newtheorem*{notation}{Notation}

\newtheorem{example}{Example}
\theoremstyle{remark}
\newtheorem{rem}{Remark}

\newcommand{\oi}{\mathbin{{\circ}_i}}

\DeclareMathOperator{\Gra}{Gra}
\def\oldvec{\mathaccent "017E\relax } 
\DeclareMathOperator{\Or}{\mathsf{O\oldvec{r}}}
\DeclareMathOperator{\Jac}{Jac}

\title[Defining properties of the Kontsevich unoriented graph complex]{The defining properties of the Kontsevich unoriented graph complex}

\author[N.\,J.\,Rutten]{Nina J. Rutten${}^{{1),2)}}$}
\thanks{${}^{{1)}}$\:\textit{Adrress}: 
Johann Bernoulli Institute for Mathematics and Computer Science, University of Groningen,
P.O.\,Box~407, 9700~AK Groningen, The Netherlands.%
}

\thanks{${}^{{2)}}$\:\textit{Present address}:
Mathematical Institute, Utrecht University,
P.O.\, Box~80010, 3508~TA  Utrecht, The Netherlands.
\quad \textit{E-mail}: \texttt{N.J.Rutten\symbol{"40}uu.nl}
}

\author[A.\,V.\,Kiselev]{Arthemy V. Kiselev${}^{{1),3)}}$}
\thanks{${}^{{3)}}$\:\textit{Present address}:
Bernoulli Institute for Mathematics, Computer Science and Artificial Intelligence, University of Groningen,
P.O.\,Box~407, 9700~AK Groningen, The Netherlands.
\quad \textit{E-mail}: \texttt{A.V.Kiselev\symbol{"40}rug.nl}%
}

\date{19 November 2018}


\begin{document}
\begin{abstract}
Consider the real vector space of formal sums of non\/-\/empty, finite unoriented graphs without multiple edges and loops. Let the vertices of graphs be unlabelled but let every graph~$\gamma$ be endowed with an ordered set of edges~$\sfE(\gamma)$. Denote by~$\Gra$ the vector space of formal sums of graphs modulo the relation
$(\gamma_1,\mathsf{E}(\gamma_1))-\text{sign}(\sigma) (\gamma_2,\mathsf{E}(\gamma_2)) = 0$ for topologically equal graphs~$\gamma_1$ and~$\gamma_2$ whose edge orderings differ by a permutation~$\sigma$. 
The zero class in~$\Gra$ is represented by sums of graphs that cancel via the above relation. 
The Lie bracket of graphs with ordered edge sets is defined using the 
insertion of a graph into vertices of the other one.
We give an explicit proof of the theorems which state that the space~$\Gra$ is a well\/-\/defined differential graded Lie algebra: both the Lie bracket $[{\cdot},{\cdot}]$ and the vertex\/-\/expanding differential~$\Id=[{\bullet}\!{-}\!{\bullet},{\cdot}]$ respect the calculus modulo zero graphs.
\end{abstract}
\maketitle

\subsection*{Introduction}
Consider the real vector space spanned by non\/-\/empty, not necessarily connected finite unoriented graphs with unlabelled vertices and without multiple edges and loops. Endow every such graph~$\gamma$ with an ordering of edges $\sfE(\gamma)=I\wedge II\wedge\ldots\wedge\#\text{Edge}(\gamma)$, and introduce the relation, 
\begin{equation}\label{EqRelation}
\bigl(\gamma_1,\sfE(\gamma_1)\bigr)=(-)^\sigma\bigl(\gamma_2,\sfE(\gamma_2)\bigr), 
\end{equation}
between topologically equal graphs~$\gamma _1$,\ $\gamma_2$ whose edge orderings differ by a permutation $\sigma\colon\text{Edge}(\gamma_1)\mapsto\text{Edge}(\gamma_2)$. Denote by~$\Gra$ the quotient of the vector space of graphs at hand 
modulo the above relation (cf.~\cite{Ascona96,GRTBV} or~\cite{JNMP}).

By definition, \emph{zero graphs} represent the zero class in~$\Gra$, that is, the (formal sums of) graphs which equal minus themselves under a symmetry that induces a parity\/-\/odd permutation of edges. For instance, the triangle with edges $I$,\ $II$,\ $III$ admits a flip $I\rightleftarrows II$, $III\rightleftarrows III$; the parity\/-\/odd cyclic permutation $I\mapsto II\mapsto III\mapsto IV\mapsto I$ of consecutive edges is a symmetry of the square. Hence these are zero graphs.

From~\cite{JPCS} and references therein we recall that the bi\/-\/linear operation of insertion~$\oi$ of a graph~$\gamma_1$ into a graph~$\gamma_2$ yields the sum of graphs~$\gamma_1\oi\gamma_2=\sum_{v\in\text{Vert}(\gamma_2)}(\gamma_1\longrightarrow v\text{ in }\gamma_2)$ such that the edges in a graph~$\gamma_2$ incident to a vertex~$v$ 
are redirected to vertices in~$\gamma_1$, each edge running over the vertices of~$\gamma_1$ by its own Leibniz rule. By definition, for each term~$g$ in $\gamma_1\oi\gamma_2$, the edge ordering is $\sfE(g)\mathrel{{:}{=}}\sfE(\gamma_1)\wedge\sfE(\gamma_2)$. Now put (extending by linearity)
\begin{equation}\label{EqBracket}
[\gamma_1,\gamma_2] \stackrel{\text{def}}{=} \gamma_1\oi\gamma_2 - 
(-)^{\#\text{Edge}(\gamma_1)\cdot\#\text{Edge}(\gamma_2)} \gamma_2\oi\gamma_1.
\end{equation}

\begin{example}\label{ExLeavesCancel}
Let us remember that $[{\bullet}\!{\stackrel{1}{-}}\!{\bullet},{\bullet}\!{\stackrel{2}{-}}\!{\bullet}] =
{\bullet}\!{\stackrel{1}{-}}\!{\bullet}\!{\stackrel{2}{-}}\!{\bullet} +
{\bullet}\!{\stackrel{2}{-}}\!{\bullet}\!{\stackrel{1}{-}}\!{\bullet} - (-)^{1\cdot 1}
\bigl({\bullet}\!{\stackrel{2}{-}}\!{\bullet}\!{\stackrel{1}{-}}\!{\bullet} +
{\bullet}\!{\stackrel{1}{-}}\!{\bullet}\!{\stackrel{2}{-}}\!{\bullet} \bigr) =
4\cdot {\bullet}\!{-}\!{\bullet}\!{-}\!{\bullet} = \boldsymbol{0}\in\Gra$
because the flip ${\bullet}\!{\stackrel{1}{-}}\!{\bullet}\!{\stackrel{2}{-}}\!{\bullet}
\mapsto {\bullet}\!{\stackrel{2}{-}}\!{\bullet}\!{\stackrel{1}{-}}\!{\bullet}$ is a symmetry which induces the parity\/-\/odd transposition of edges.
\end{example}

\begin{claimNN} The bi\/-\/linear bracket~$[{\cdot},{\cdot}]$ is well defined on~$\Gra$ because $[\text{zero graph},\text{any graph}] = (\text{sum of zero graphs}) + 0\cdot(\text{sum of graphs})$.\\
$\bullet$\quad The differential $\Id({\cdot})=[{\bullet}\!{-}\!{\bullet},\cdot]$ is well defined on~$\Gra$ because $\Id(\text{zero graph})=\boldsymbol{0}\in\Gra$ by the above and~$\Id^2(\cdot)=0$ in the space of graphs.\\
$\bullet$\quad The graded Jacobi identity holds for the graded skew\/-\/symmetric bracket~$[{\cdot},{\cdot}]$ on~$\Gra$. Consequently, the Lie bracket of $\Id$-\/cocycles is a $\Id$-\/cocycle as well.
\end{claimNN}

Summarizing, the space~$\Gra$ of graphs with unlabelled vertices and ordered sets of edges is a differential graded Lie algebra~(dgLa). By inspecting the formation of the respective right\/-\/hand sides in the brackets like $\text{zero graph},\cdot]$, etc., we provide an explicit proof of the defining properties for this dgLa structure.

The knowledge of cocycles in the Kontsevich unoriented graph complex~$(\Gra,\Id)$ is important because under the orientation morphism~$\Or$,
the $\Id$-\/cocycles on $n$~vertices and $2n-2$ edges are mapped to the universal deformations of Poisson brackets on finite\/-\/dimensional affine manifolds (see~\cite{Ascona96} or~\cite{JNMP,SQS17,JPCS} and~\cite{Or2018}). For example, the tetrahedron~$\boldsymbol{\gamma}_3\in\ker\Id$ from~\cite{Ascona96},
as well as the pentagon\/-{} and heptagon\/-\/wheel cocycles~$\boldsymbol{\gamma}_5$,\ $\boldsymbol{\gamma}_7\in\ker\Id$ are presented in~\cite{JNMP} (see also references therein). The Poisson structure symmetry which corresponds to~$\boldsymbol{\gamma}_5$ is found in~\cite{SQS17,JPCS}. Likewise, the symmetry built from~$\boldsymbol{\gamma}_7$ is reported in~\cite{Or2018}.

This paper is structured as follows. After recalling some definitions and notation, we illustrate their work by proving the identity $\Id^2=0$ for the vertex blow\/-\/up $\Id=[\edge,{\cdot}]$. We show that the right\/-\/hand side of the formula $\Id^2(\gamma)=0$ vanishes identically for any graphs~$\gamma$, i.e.\ not just being equal to a sum of zero graphs in~$\Gra$. (A proof of that weaker statement, $\Id^2(\gamma)=\boldsymbol{0}\in\Gra$,
by using the Jacobi identity for the bracket~$[{\cdot},{\cdot}]$ is a corollary to Theorems~\ref{TheoBracketZero} and~\ref{ThJacobi}.) 
The main fact by which the space~$\Gra$ is well defined is Theorem~\ref{TheoBracketZero} on p.~\pageref{TheoBracketZero} stating that the bracket $[\gamma^\zstroke,{\cdot}]$ with a zero graph~$\gamma^\zstroke$ vanishes in~$\Gra$. Its proof relies on the technique of orbits for symmetry group action on vertices and edges of zero graphs.
Finally, in Theorem~\ref{ThJacobi} we present the graded Jacobi identity~\eqref{EqGradJac} for~$[{\cdot},{\cdot}]$; the right\/-\/hand side of~\eqref{EqGradJac} vanishes identically through cancellations in the vector space of graphs. The Jacobi identity implies in a standard way that the Lie bracket of two $\Id$-\/cocycles is a $\Id$-\/cocycle as well.

We follow the notation and conventions which are adopted in~\cite{JPCS} (and in~\cite{JNMP} where examples are given with practical calculations in the Kontsevich graph complex).
%
We refer to~\cite{Or2018} or~\cite{GRTBV} for more details about the complex~$\Gra$ (introduced in~\cite{Ascona96}).

\begin{notation}
A sum of graphs in the space~$\Gra$ of formal sums of graphs consisting of one term is called a \emph{single graph}. By convention, a single graph~$\gamma^x$ has $n_x$ vertices and $k_x$ edges. Denote by $\sfV(\gamma^x)$ the set of vertices, by $\sfE(\gamma^x)$ the set of edges, and by $X_{\gamma^x} \mathrel{{:}{=}} 
\sfV(\gamma^x)\sqcup \sfE(\gamma^x)$ the set of vertices and edges of a single graph~$\gamma^x$. Unless stated otherwise, the wedge ordered set of edges is given by 
$\sfE(\gamma^x) = I^{(x)}\wedge II^{(x)}\wedge \cdots \wedge K^{(x)}$; any arbitrarily chosen labelling of the vertices is denoted by $1^{(x)}$, $2^{(x)}$, $\dots$, $n_x^{(x)}$.
Let $v$~be a vertex of a graph~$\gamma^x$. We denote by~$N(v)$ the set of neighbouring vertices of~$v$ and by~$\bar{N}(v)$ the set of edges attached to~$v$. For a given number $n\in\mathbb{N}\cup\{0\}$ we denote by $\Pi_n(v)$ the set of all possible ordered partitions of the set of incident edges~$\bar{N}(v)$ into $n$~disjoint sets (possibly, empty). Any element of~$\Pi_n(v)$ is of the form $\pi = (S_1$, $S_2$, $\cdots$, $S_n)$ such that $S_1\sqcup S_2\sqcup\cdots\sqcup S_n = \bar{N}(v)$.
\end{notation}

Now let us rephrase the definition of a symmetry of a graph in terms of the sets of edges attached to a fixed vertex. (This construction 
is equivalent to the definition of a graph automorphism.)

\begin{define}
Let $\gamma$~be a single graph in~$\Gra$. A \emph{symmetry} of the graph~$\gamma$ is a permutation $\sigma\colon X_{\gamma}\rightarrow X_{\gamma}$ such that $\sigma(\sfV(\gamma))=\sfV(\gamma)$ and $\sigma(\sfE(\gamma))=\sfE(\gamma)$ as sets, and such that $\sigma(\bar{N}(v)) = \bar{N}(\sigma(v))$ for all vertices~$v$ of~$\gamma$.
\end{define}

\begin{notation}
For a given symmetry~$\sigma$ of a single graph~$\gamma$, we denote by $\sigma_V \colon \sfV(\gamma)\rightarrow\sfV(\gamma)$ the permutation of vertices induced by~$\sigma$. Likewise, we denote by $\sigma_E \colon \sfE(\gamma)\rightarrow\sfE(\gamma)$ the permutation of edges induced by~$\sigma$.
\end{notation}

We recall from \cite{JNMP} that, by definition, the insertion~$\oi$ of a sum of graphs into another sum of graphs is linear with respect to both arguments. Therefore it suffices to define the insertion of a single graph in~$\Gra$ into another single graph in~$\Gra$.

\begin{define}
Consider two single graphs~$\gamma^1$ and~$\gamma^2$ in~$\Gra$. The $\gamma^1$-\emph{blow\/-\/up} $B_{\gamma^1}(v)$ of a vertex~$v$ in~$\gamma^2$ is the following sum of graphs:
\[
B_{\gamma^1}(v) \mathrel{{:}{=}} 
\sum_{\pi\in\Pi_{n_1}(v)} \gamma^2_{v,\pi},
\]
where $\gamma^2_{v,\pi}$ denotes the graph obtained from the graph~$\gamma^2$ by replacing the vertex~$v$ in it by the entire graph~$\gamma^1$ and where the incident edges~$\bar{N}(v)$ are reattached to the vertices in the inserted graph~$\gamma^1$ (via the correspondence given by a vertex labelling from~$\gamma^1$ and an 
ordered partition $\pi = (S_1$, $S_2$, $\dots$, $S_{n_1})$ of the set of edges~$\bar{N}(v)$ in~$\gamma^2$). 
That is, for each vertex~$w$ in~$\gamma^1$, the edges in the set~$S_w$ in the partition~$\pi$  are reattached to the vertex~$w$ in the new graph~$\gamma^2_{v,\pi}$.
By definition, the \emph{insertion} $\gamma^1\circ_i\gamma^2$ of a graph~$\gamma^1$ into a graph~$\gamma^2$ is the sum
\[
\gamma^1\circ_i\gamma^2 \mathrel{{:}{=}} 
\sum_{v\in\gamma^2} B_{\gamma^1}(v) = \sum_{v\in\gamma^2} \sum_{\pi\in\Pi_{n_1}(v)} \gamma^2_{v,\pi}.
\]
By convention, in every term of the sum $\gamma^1\circ_i\gamma^2$, the edge ordering is $\sfE(\gamma^1)\wedge\sfE(\gamma^2)$.
\end{define}

\section{The differential on Gra} 

\begin{define}
Let $\gamma$~be a graph. The edges with a valency one vertex attached to them are called \emph{leaves} of the graph~$\gamma$. Next, let $\gamma$~be a sum of graphs in the space~$\Gra$. The graphs in the sum~$\Id(\gamma)$ where the new edge~$\edge$ appears as a leaf are \emph{leaved graphs}.
\end{define}

By default, we shall consider not necessarily connected but always leafless, i.e.\ the ones without leaves (cf.~\cite{Ascona96,GRTBV}.

\begin{proposition}
\label{Propddzero}
For any (sum of) single graphs $\gamma\in\Gra$ without leaves, the linear operator~$\Id$ which acts by the formula
\begin{equation}\label{d}
\Id(\gamma) \mathrel{{:}{=}} 
[\edge,\gamma] = \edge\circ_i\gamma - (-)^{\#\sfE(\edge)\#\sfE(\gamma)} \gamma \circ_i \edge
\end{equation}
is a differential: $\Id \circ \Id =0$,
where the right\/-\/hand side is the zero element of the vector space of graphs.
\end{proposition}


\begin{lemma}\label{LeavesCancel}
Let $\gamma$~be a sum of leafless graphs in the space~$\Gra$. All leaved graphs in~$\Id(\gamma)$ cancel.
\end{lemma}

\begin{proof}
Again, since $\Id$~is linear it suffices to assert the statement for single graphs. 
We consider the edge~$\edge$ as the first graph. 
Denote by~$E^{(1)}$ the edge in the graph~$\edge$, then the singleton set $\sfE(\edge)=E^{(1)}$ is wedge ordered.
We may introduce any labelling of vertices for the edge graph~$\edge$; 
let it be $1^{(1)}$, $2^{(1)}$.
Then the graph~$\edge$ is represented by $ _{1^{(1)}}\!\bullet\!\!\!\frac{\,\,E^{(1)}}{}\!\!\!\bullet\!_{2^{(1)}}$.

Now let $\gamma^2$~be a single graph in the space~$\Gra$.
We can represent each blow\/-\/up $B_{\gamma^2}(i^{(1)})$ of a vertex~$i$ in~$\edge$ by $_{i^{(1)}}\circled{$\gamma^2$}\!\frac{\quad}{}\!\!\!\bullet\!_{j^{(1)}}$, 
where $j=3-i$ (the vertex other than~$i$).
Now the subtrahend, $(-)^{\#\sfE(\edge)\#\sfE(\gamma^2)}\gamma^2\circ_i\edge$ 
in equation~\eqref{d} --\,applied to the graphs under study\,-- 
is 
expressed as follows:
\begin{align*}
(-)^{\#\sfE(\edge)\#\sfE(\gamma^2)} \gamma^2 \circ_i \edge &= 
(-)^{k_2} \bigl( B_{\gamma^2} (1^{(1)}) + B_{\gamma^2} (2^{(1)}) \bigr)\\
&= (-)^{k_2}\,\, _{i^{(1)}}\circled{$\gamma^2$}\!\frac{\quad}{}\!\!\!\bullet\!_{j^{(1)}} + (-)^{k_2}\,\, _{i^{(1)}}\!\!\bullet\!\!\!\frac{\quad}{}\!\circled{$\gamma^2$}_{j^{(1)}}.
\end{align*}
It is clear that this 
sum consists only of leaved graphs.
The edge ordering of all graphs in this sum is given by $
I^{(2)} \wedge II^{(2)} \wedge \cdots \wedge K^{(2)} \wedge E^{(1)}$. Note that graphs with exactly the same topology appear in the minuend $\edge\circ_i\gamma^2$ of equation~\eqref{d} as well. 
Namely, such are 
the graphs of the form $\gamma^2_{v,(\bar{N}(v),\varnothing 
)}$ and $\gamma^2_{v,(\varnothing 
,\bar{N}(v))}$, all with the edge ordering $E^{(1)} \wedge I^{(2)} \wedge II^{(2)} \wedge \cdots \wedge K^{(2)}$. 
We remark that in the minuend, there 
appear no leaved graphs other than these, and that 
\begin{align*}
(-)^{k_2} I^{(2)} \wedge II^{(2)} \wedge \cdots \wedge K^{(2)} \wedge E^{(1)} &= (-)^{k_2} (-)^{k_2} E^{(1)} \wedge I^{(2)} \wedge II^{(2)} \wedge \cdots \wedge K^{(2)}\\
&= E^{(1)} \wedge I^{(2)} \wedge II^{(2)} \wedge \cdots \wedge K^{(2)}.
\end{align*}
For all the leaved graphs appearing in the differential~$\Id(\gamma^2)$, see equation~\eqref{d}, we obtain
\[
\sum_{v\in\gamma^2} (\gamma^2_{v,(\bar{N}(v),\varnothing 
)} + \gamma^2_{v,(\varnothing 
,\bar{N}(v))}) - (-)^{k_2} (-)^{k_2} \sum_{v\in\gamma^2} (\gamma^2_{v,(\bar{N}(v),\varnothing 
)} + \gamma^2_{v,(\varnothing 
,\bar{N}(v))}) = 0.
\]
This shows that 
all the leaved graphs cancel in~$\Id(\gamma^2)$.
\end{proof}


\begin{cor}\label{d2}
The differential~$\Id$ applied to a single graph~$\gamma\in\Gra$ without leaves is equal to
\[
\Id(\gamma) = \sum_{v\in\sfV(\gamma)} \sum_{\pi\in \tilde{\Pi}_{2}(v)} \gamma_{v,\pi},
\]
where $\tilde{\Pi}_{2}(v)$~is the set of ordered partitions of~$\bar{N}(v)$ into two \emph{non\/-\/empty} sets.
\end{cor}

For the proof of Proposition~\ref{Propddzero} it is convenient to make a distinction between the two kinds of graphs that can appear in the sum $\Id\circ\Id(\gamma)$ 
obtained by applying the composition of operators $\Id\circ\Id$ to a single graph~$\gamma$ in~$\Gra$.

\begin{define}
Let $\gamma^2$~be a single graph in~$\Gra$ and let $v$~be one of its vertices. Let $\gamma^2_{v,\pi}$ be a graph appearing in the edge\/-\/blow\/-\/up $B_{\edge}(v)$ of~$v$, that is, in a sum of graphs that contributes to~$\Id(\gamma)$. We introduce the new labellings $_{1^{(0)}}\!\bullet\!\!\!\frac{\,\,E^{(0)}}{}\!\!\!\bullet\!_{2^{(0)}}$ for the vertices and edges of the edge~$\edge$ that is inserted when we apply the differential~$\Id$ for the second time, so that we can distinguish it from the edge that was inserted firstly. (We choose the zero labellings because it induces a natural order of labellings in the edge ordering of any graph~$\gamma$ in $\Id\circ\Id(\gamma^2)$, namely $\sfE(\gamma) = E^{(0)}\wedge E^{(1)}\wedge I^{(2)}\wedge II^{(2)}\cdots\wedge K^{(2)}$.)

We say that graphs in the sum $\Id\circ\Id(\gamma^{(2)})$ where the second edge\/-\/blow\/-\/up is applied to an \emph{old} vertex~$w$ in $\gamma^2_{v,\pi}$ (i.e.\ neither of the vertices $1^{(1)}$ or $2^{(1)}$) are \emph{distant blown\/-\/up graphs}. 
These graphs are of the form $(\gamma^2_{v,\pi})_{w,\tau}$, where $\pi=(S_1,S_2)\in\tilde{\Pi}_{2}(v)$ and $\tau=(T_1,T_2)\in\tilde{\Pi}_{2}(w)$.

In turn, we say that graphs in the sum $\Id\circ\Id(\gamma^{(2)})$, where the second edge\/-\/blow\/-\/up is applied to one of the two \emph{new} vertices, $1^{(1)}$ or $2^{(1)}$ in $\gamma^2_{v,\pi}$, are \emph{nested blown\/-\/up graphs}. 
These graphs are of the form $(\gamma^2_{v,\pi})_{i^{(1)},\tau}$, for $i=1$ or $i=2$, where $\pi=(S_1,S_2)\in\tilde{\Pi}_{2}(v)$ and $\tau=(T_1,T_2)\in\tilde{\Pi}_{2}(i^{(1)})$.
\end{define}

Now we have the tools to prove Proposition~\ref{Propddzero}.

\begin{proof}[Proof of Proposition~\protect\ref{Propddzero}]
It suffices to verify 
the claim 
for single graphs in~$\Gra$ and for the realization of~$\Id$ as given in Corollary~\ref{d2}. 
In order to prove that $\Id\circ\Id(\gamma)$ vanishes for any single graph~$\gamma\in\Gra$, we show 
first that all distant blown\/-\/up graphs cancel in~$\Id\circ\Id(\gamma)$, and second, that all nested blown\/-\/up graphs cancel in~$\Id\circ\Id(\gamma)$.
Let $\gamma^2$~be a single graph in the space~$\Gra$ and consider the sum $\Id\circ\Id(\gamma^2)$.
Let $\gamma^a \mathrel{{:}{=}} 
(\gamma^2_{v,\pi})_{w,\tau}$ be a distant blown\/-\/up graph in this sum. Here a vertex~$v$ is blown up first and a vertex~$w$ is blown up second, so $\gamma^a$ is obtained in the sum~$\Id(B_\edge(v))$. Then locally, near the newly inserted edges, the graph~$\gamma^a$ looks as follows 
\[
^{A}_{1^{(1)}}\!\!\bullet\!\!\!\frac{\,\,\,\,E^{(1)\,}}{}\!\!\!\bullet^{\,\,B}_{2^{(1)}} \qquad ^{C}_{1^{(0)}}\!\!\bullet\!\!\!\frac{\,\,\,\,E^{(0)\,}}{}\!\!\!\bullet^{\,\,D}_{2^{(0)}},
\]
where the sets $A$, $B$, $C$, $D$ next to the vertices are the sets of edges attached to the respective vertices: by the definition of a blow\/-\/up we have that $A\cup B = \bar{N}(v)$ and $C\cup D = \bar{N}(w)$. (Note that the sets~$\bar{N}(v)$ and~$\bar{N}(w)$ intersect if and only if the vertices~$v$ and~$w$ are neighbours in the graph~$\gamma^2$.)

We pair the graph $\gamma^a = (\gamma^2_{v,\pi})_{w,\tau}$ with the graph $(\gamma^2_{w,\tau})_{v,\pi} \mathrel{{=}{:}} 
\gamma^b$ in~$\Id(B_\edge(w))$. Locally, near the inserted edges, the graph~$\gamma^b$ looks like this:
\[
^{A}_{1^{(0)}}\!\!\bullet\!\!\!\frac{\,\,\,\,E^{(0)\,}}{}\!\!\!\bullet^{\,\,B}_{2^{(0)}} \qquad ^{C}_{1^{(1)}}\!\!\bullet\!\!\!\frac{\,\,\,\,E^{(1)\,}}{}\!\!\!\bullet^{\,\,D}_{2^{(1)}}.
\]
Note that a distant blown\/-\/up graph is matched with 
precisely one other distant blown\/-\/up graph via this pairing,\footnote{\label{Setsoffour} The graphs~$\gamma^a$ and~$\gamma^b$ appear in sets of four in~$\Id\circ\Id(\gamma^2)$. Namely, the edges can both be inserted in two ways -- still producing exactly the same graph~-- by exchanging the roles of vertex~$1^{(i)}$ and~$2^{(i)}$ for~$i=0,1$, that is, by swapping the set~$A$ with the set~$B$ (for~$i=1$) and swapping the set~$C$ with the set~$D$ (for~$i=0$).}
and the set of all blow\/-\/ups is a disjoint union of such pairs.

Since the labelling of vertices can be chosen arbitrarily, the graphs~$\gamma^a$ and~$\gamma^b$ are topologically equal. Specifically, they are equal under the 
matching of edges: $E^{(0)(a)} = E^{(1)(b)}$, $E^{(1)(a)} = E^{(0)(b)}$, and $\#^{(2)(a)} = \#^{(2)(b)}$ for all other edges~$\#^{(2)(a)}$ of~$\gamma^a$ and~$\#^{(2)(b)}$ of~$\gamma^b$ (since the sets $A$, $B$, $C$, $D$ are fixed).
Then the wedge ordered set of edges~$\sfE(\gamma^a)$ of~$\gamma^a$ satisfies the equality
\begin{align*}
\sfE(\gamma^a) &= E^{(0)(a)} \wedge E^{(1)(a)} \wedge I^{(2)(a)} \wedge II^{(2)(a)}\wedge\cdots\wedge K^{(2)(a)}\\
&= E^{(1)(b)} \wedge E^{(0)(b)} \wedge I^{(2)(b)} \wedge II^{(2)(b)}\wedge\cdots\wedge K^{(2)(b)}\\
&= -E^{(0)(b)} \wedge E^{(1)(b)} \wedge I^{(2)(b)} \wedge II^{(2)(b)}\wedge\cdots\wedge K^{(2)(b)}\\
&= -\sfE(\gamma^b).
\end{align*}
Hence the graph~$\gamma^a$ is cancelled by the graph~$\gamma^b$. 
All \emph{distant} blown\/-\/up graphs thus cancel in~$\Id\circ\Id(\gamma^2)$.

Let $\gamma^c \mathrel{{:}{=}} 
(\gamma^2_{v,\pi})_{i^{(1)},\tau}$ be a nested blown\/-\/up graph in the sum $\Id\circ\Id(\gamma^2)$, for $i=1$ or~$i=2$. Here a vertex~$v$ is blown up first and the vertex $i^{(1)}$, produced by the first blow up, is blown up second. Without loss of generality, say~$i=1$. Then the graph~$\gamma^c$ is obtained in the sum $\Id(B_\edge(v))$. Locally, near the newly inserted edges, the graph~$\gamma^c$ looks as follows: 
\[
^{A}_{1^{(0)}}\!\!\bullet\!\!\!\frac{\,E^{(0)\,\,\,\,\,}}{}\!\!\!\!\!\!^{\,\,B}_{2^{(0)}}\!\!\!\bullet\!\!\!\frac{\,\,\,E^{(1)\,}}{}\!\!\!\bullet^{\,\,C}_{2^{(1)}},
\]
where the sets $A$, $B$, $C$ next to the vertices are the sets of edges attached to the respective vertices. By the definition of a blow-up we have that 
\[
\bar{N}(v) = A\sqcup B \sqcup C, \qquad
\bar{N}(1^{(0)}) = A\sqcup \{2^{(0)}\}, \qquad
\bar{N}(2^{(0)}) = B\sqcup \{1^{(0)},2^{(1)}\}, \qquad
\bar{N}(2^{(1)}) = C\sqcup \{2^{(0)}\}.
\]
We can rewrite the ordered partitions~$\pi$ and~$\tau$ as 
\[
\pi=(A\sqcup B,C)\quad \text{and} \quad \tau = (A,B\sqcup\{E^{(1)}\}).
\]
Now we pair the graph $\gamma^c=(\gamma^2_{v,\pi})_{i^{(1)},\tau}$ with the graph $\gamma^d
\mathrel{{=}{:}} 
(\gamma^2_{v,\pi'})_{2^{(1)},\tau'}$, where the ordered partitions are given by
\[
\pi' = (A,B\sqcup C)\quad \text{and}\quad \tau'= (\{E^{(1)}\}\sqcup B,C).
\]
Locally, near the inserted edges, the graph~$\gamma^b$ looks as follows,
\[
^{A}_{1^{(1)}}\!\!\bullet\!\!\!\frac{\,\,E^{(1)\,\,\,}}{}\!\!\!\!\!\!^{\,\,B}_{1^{(0)}}\!\!\!\bullet\!\!\!\frac{\,\,\,E^{(0)\,}}{}\!\!\!\bullet^{\,\,C}_{2^{(0)}}.
\]
Since the labelling of vertices can be chosen arbitrarily, the graphs~$\gamma^c$ and~$\gamma^d$ are topologically equal under the 
matching of edges: $E^{(0)(c)} = E^{(1)(d)}$, $E^{(1)(c)} = E^{(0)(d)}$, and $\#^{(2)(c)} = \#^{(2)(d)}$ for all other edges~$\#^{(2)(c)}$ of~$\gamma^c$ and~$\#^{(2)(d)}$ of~$\gamma^d$ (since the sets $A$, $B$, $C$ are fixed).
Then the wedge ordered set of edges~$\sfE(\gamma^c)$ of $\gamma^c$ satisfies the equality similar to the one we have seen before,
\begin{align*}
\sfE(\gamma^c) &= E^{(0)(c)} \wedge E^{(1)(c)} \wedge I^{(2)(c)} \wedge II^{(2)(c)}\wedge\cdots\wedge K^{(2)(c)}\\
&= E^{(1)(d)} \wedge E^{(0)(d)} \wedge I^{(2)(d)} \wedge II^{(2)(d)}\wedge\cdots\wedge K^{(2)(d)}\\
&= -E^{(0)(d)} \wedge E^{(1)(d)} \wedge I^{(2)(d)} \wedge II^{(2)(d)}\wedge\cdots\wedge K^{(2)(d)}\\
&= -\sfE(\gamma^d).
\end{align*}
Hence the graph~$\gamma^c$ is cancelled by the graph~$\gamma^d$. 
Thus all \emph{nested} blown\/-\/up graphs cancel in~$\Id\circ\Id(\gamma^2)$.
The proof is complete.
\end{proof}

\section{The bracket with a zero graph}
\noindent%
We begin with recalling necessary 
properties of group actions. 

\begin{rem}\label{Act}
Let $\gamma\in\Gra$~be a zero graph, i.e.\ it suppose has a symmetry given by a permutation~$\sigma$ of vertices and edges that induces a parity odd permutation of edges. Since the permutation~$\sigma$ acts\footnote{Any permutation~$\sigma$ that defines a symmetry of a graph~$\gamma$ acts on the space~$X_{\gamma}$ of vertices and edges in a particular way: vertices are sent to vertices and edges are sent to edges. Hence we could also consider~$\sigma$ acting on the product space $\sfV (\gamma) \times \sfE (\gamma)$ of the sets of vertices and 
edges.%
}
on the set~$X_{\gamma}$ of vertices and edges of the graph~$\gamma$, 
it acts\footnote{\label{ActConfusing}%
A permutation~$\sigma$ acts here on the power set $P(X_\gamma)$ in the sense that the power set~$P(X_\gamma)$ itself stays 
invariant under~$\sigma$, not the individual elements of the power set. In turn, the permutation~$\sigma$ acts on the product set $P(X_{\gamma}) \times \cdots \times P(X_{\gamma})$ in the sense that this product set is 
invariant under~$\sigma$, not the elements of the product set.}
naturally on the power set $P(X_{\gamma})$ of $X_{\gamma}$, namely as follows: $\sigma(S) = \{\sigma(s):s\in S\}$ for~$S\in P(X_{\gamma})$. Hence 
$\sigma$~acts 
naturally on any finite product space of the power set~$P(X_{\gamma})$:
indeed, 
$\sigma\bigl(\,(S_1$,$S_2$,$\cdots$,$S_n)\,\bigr)=(\sigma(S_1)$,$\sigma(S_2)$,$\cdots$,$\sigma(S_n))$ for $(S_1$,$S_2$,$\cdots$,$S_n))\in P(X_{\gamma}) \times \cdots \times P(X_{\gamma})$. This enables us to speak of the orbit $O_{(S_1,S_2,\cdots,S_n)}$ of an ordered set $(S_1$,$S_2$,$\cdots$,$S_n)\in P(X_{\gamma})\times \cdots \times P(X_{\gamma})$.
\end{rem}

\begin{theor}\label{TheoBracketZero}
Let $\gamma^\zstroke$ and $\gamma^2$ be formal sums of graphs in~$\Gra$ and let $\gamma^\zstroke$~be a sum of zero graphs. Then the Lie bracket $[\gamma^\zstroke,\gamma^2]$ vanishes in~$\Gra$.
\end{theor}


\begin{proof}
By linearity of the Lie bracket, it suffices to show the statement for sums of graphs consisting of one term. Let $\gamma^\zstroke$ and $\gamma^2$~be single graphs 
in the space~$\Gra$.
Let $\gamma^\zstroke$~be a zero graph.
Then this graph has a symmetry, a permutation~$\sigma$ of edges and vertices that induces a parity\/-\/odd permutation~$\sigma_E$ of edges. 
The Lie bracket of~$\gamma^\zstroke$ and~$\gamma^2$ can be expressed in terms of blow\/-\/ups: 
\begin{equation}\label{Bracket}
[\gamma^\zstroke,\gamma^2] = \gamma^\zstroke \circ_i \gamma^2 - (-)^{\# \sfE (\gamma^\zstroke)\# \sfE (\gamma^2)} \gamma^2 \circ_i \gamma^\zstroke
= \sum_{v\in \sfV (\gamma^2)} B_{\gamma^\zstroke}(v) - (-)^{k_z k_2} \sum_{v\in \sfV (\gamma^\zstroke)} B_{\gamma^2}(v).
\end{equation}

\textbf{Part~1.} 
Let us prove that the subtrahend $(-)^{\# \sfE (\gamma^\zstroke)\# \sfE (\gamma^2)} \sum_{v\in \sfV (\gamma^\zstroke)} B_{\gamma^2}(v)$ is equal to zero in~$\Gra$.


Let $v_1$, $v_2$ be vertices in the graph~$\gamma^\zstroke$ and denote by~$O_{v_1}$ and~$O_{v_2}$ their respective orbits under~$\sigma_V$. 
It is standard that if two orbits~$O_{v_1}$ and~$O_{v_2}$ intersect, then they coincide. 
Denote by~$\mathbf{O}^\zstroke$ the set of all orbits of vertices in~$\gamma^\zstroke$ under~$\sigma_V$. Note that the set~$\sfV(\gamma^\zstroke)$ of vertices in~$\gamma^\zstroke$ is equal to the disjoint union of orbits in~$\mathbf{O}^\zstroke$. Then we can rewrite the subtrahend in equation (\ref{Bracket}) as follows,
\[
(-)^{k_z k_2} \sum_{v\in \sfV (\gamma^\zstroke)} B_{\gamma^2}(v) = (-)^{k_z k_2} \sum_{O\in \mathbf{O}^\zstroke} \sum_{v\in O} B_{\gamma^2}(v).
\]
Hence, in order to show that the subtrahend vanishes, it suffices to show that for each orbit~$O\in\mathbf{O}^\zstroke$ the sum of $\gamma^2$-\/blow\/-\/ups $\sum_{v\in O} B_{\gamma^2}(v)$ is equal to zero in the space~$\Gra$.
Let $O\in \mathbf{O}^\zstroke$~be the orbit of a vertex under~$\sigma_V$; denote by~$\ell$ the length of this orbit. 
Let $v_\circ$~be a vertex in the orbit~$O$. Then $O$~can be seen as the orbit~$O_{v_\circ}$ of~$v_\circ$. We can rewrite and expand the sum of $\gamma^2$-\/blow\/-\/ups in the following way,
\begin{equation}\label{Sums}
\sum_{v\in O} B_{\gamma^2}(v) = \sum_{i=0}^{\ell-1} B_{\gamma^2}(\sigma^i(v_\circ)) = \sum_{i=0}^{\ell-1} \sum_{\pi\in\Pi_{n_2}(\sigma^i(v_\circ))} \gamma^\zstroke_{\sigma^i(v_\circ),\pi}.
\end{equation}
By the definition of a symmetry and by the properties of orbits we have that $\Pi_{n_2}(\sigma^i(v_\circ)) = \{\sigma^i(\pi) : \pi\in\Pi_{n_2}(v_\circ)\}$, where $\sigma$~acts on the ordered set~$\pi$ like it does on ordered sets in Remark~\ref{Act}. Using this, we 
rewrite equation (\ref{Sums}) once again:
\begin{equation}\label{TheSum}
\sum_{v\in O} B_{\gamma^2}(v) = \sum_{i=0}^{\ell-1} \sum_{\pi\in\Pi_{n_2}(\sigma^i(v_\circ))} \gamma^\zstroke_{\sigma^i(v_\circ),\pi} = \sum_{\pi\in\Pi_{n_2}(v_\circ)} \sum_{i=0}^{\ell-1} \gamma^\zstroke_{\sigma^i(v_\circ),\sigma^i(\pi)}.
\end{equation}
In order to show that the right hand side of the above equation vanishes, it suffices to show that 
\begin{equation}\label{EvenSum}
\sum_{i=0}^{\ell-1} \gamma^\zstroke_{\sigma^i(v_\circ),\sigma^i(\pi)} = 0,
\end{equation}
for all $\pi\in\Pi_{n_2}(v_\circ)$.
Let $\pi_\circ \mathrel{{:}{=}} 
(S_1, S_2,\cdots, S_{n_2})$ be an ordered partition in~$\Pi_{n_2}(v_\circ)$. Consider 
two consecutive 
terms $\gamma^\zstroke_{v_\circ,\pi_\circ} \mathrel{{=}{:}} \gamma^a$ and $\gamma^\zstroke_{\sigma(v)_\circ,\sigma(\pi_\circ)} \mathrel{{=}{:}} \gamma^b$ in the above sum.
By definition of a symmetry we have that $\gamma^a$~is topologically equal to~$\gamma^b$ under the 
matching of edges $I^{(z)(a)} = \sigma_E(I^{(\zstroke)})^{(b)}$, $II^{(z)(a)} = \sigma_E(II^{(\zstroke)})^{(b)}$, $\dots$, $K^{(z)(a)} = \sigma_E(K^{(\zstroke)})^{(b)}$ and $I^{(2)(a)} = I^{(2)(b)}$, $II^{(2)(a)} = II^{(2)(b)}$, $\dots$, $K^{(2)(a)} = K^{(2)(b)}$.
Since $\sigma_E$~is a parity\/-\/odd permutation, the wedge ordered set of edges of~$\gamma^a$ satisfies the chain of equalities
\begin{align*}
\sfE(\gamma^a) &= I^{(z)(a)} \wedge II^{(z)(a)} \wedge \dots \wedge K^{(z)(a)} \wedge I^{(2)(a)} \wedge II^{(2)(a)} \wedge \dots \wedge K^{(2)(a)}\\
&= \sigma_E(I^{(\zstroke)})^{(b)} \wedge \sigma_E(II^{(\zstroke)})^{(b)} \wedge \dots \wedge \sigma_E(K^{(\zstroke)})^{(b)} \wedge I^{(2)(b)} \wedge II^{(2)(b)} \wedge \dots \wedge K^{(2)(b)}\\
&= - I^{(z)(b)} \wedge II^{(z)(b)} \wedge \dots \wedge K^{(z)(b)} \wedge I^{(2)(b)} \wedge II^{(2)(b)} \wedge \dots \wedge K^{(2)(b)}\\
&= - \sfE (\gamma^b).
\end{align*}
Therefore, $\gamma^\zstroke_{v_\circ,\pi_\circ} = \gamma^a = - \gamma^b = - \gamma^\zstroke_{\sigma(v_\circ),\sigma(\pi_\circ)}$. Since $v_\circ\in O$~is chosen arbitrarily and since the elements in~$O$ can be written as~$\sigma^i(v_\circ)$ for $0\leqslant i <\ell$, it follows that
\begin{equation}\label{Odd}
\gamma^\zstroke_{\sigma^i(v_\circ),\sigma^i(\pi_\circ)} = - \gamma^\zstroke_{\sigma^{i+1}(v_\circ),\sigma^{i+1}(\pi_\circ)}.
\end{equation}

Assume $\ell$~is even. Let us make a (disjoint) pairing of all graphs in equation~\eqref{EvenSum} and show that all pairs cancel. Specifically, we pair all graphs in Eq.~\eqref{EvenSum} as follows: $\gamma^\zstroke_{\sigma^i(v_\circ),\sigma^i(\pi_\circ)}$ is paired with $\gamma^\zstroke_{\sigma^{i+1}(v_\circ),\sigma^{i+1}(\pi_\circ)}$ for $i=0$,$2$,$4$,$\dots$,$\ell-2$. All these pairs cancel by~
\eqref{Odd}, so~$\sum_{v\in O} B_{\gamma^2}(v)=0$.

Assume $\ell$~is odd. Now we cannot organize all graphs in disjoint cancelling pairs since the 
sum~\eqref{EvenSum} consists of an odd number 
of terms. Let us 
rewrite equation~\eqref{TheSum} in such a way that we sum over the orbits of ordered sets. If the orbits have even length, we can organize the graphs into cancelling pairs again, as in the above case. If the orbits have odd length, the graphs turn out to be zero graphs.

In~\eqref{TheSum}, $\sigma$~determines a parity\/-\/odd permutation of edges, and 
$\sigma^\ell(v_\circ)=v_\circ$ for a given vertex~$v_\circ$. 
Now consider the orbit~$O_{\pi_circ}$ of an ordered partition~$\pi_circ$ 
of edges incident to~$v_\circ$ under the permutation~$\sigma^\ell$; denote by~$\ell_{\pi_\circ}$ the length of~$O_{\pi_circ}$. Also, denote by~$\lambda$ the length of the orbit~$O_{(v_\circ,\pi_\circ)}$ of the ordered set~$(v_\circ, \pi_\circ) \mathrel{{:}{=}} 
(\{v_\circ\}$, $S_1$, $S_2$, $\cdots$, $S_{n_2})$, see Remark~\ref{Act}
on p.~\pageref{Act} above.
Note that $\ell_{\pi_\circ}=\lambda/\ell$. Indeed, 
$\lambda$~is a multiple of~$\ell$ since one must have $\sigma^\lambda(v_\circ)=v_\circ$ in order to have that $\sigma^\lambda ((v_\circ,\pi_\circ))=(v_\circ,\pi_\circ)$.

First, assume $\lambda$~is odd. Then we claim that $\gamma^\zstroke_{v_\circ,\pi_\circ}$~is a zero graph. Namely, 
\[
\gamma^\zstroke_{v_\circ,\pi_\circ} = \gamma^\zstroke_{\sigma^\lambda(v_\circ),\sigma^\lambda(\pi_\circ)} = - \gamma^\zstroke_{v_\circ,\pi_\circ},
\]
where the second equality follows from equation~\eqref{Odd}.

Now assume $\lambda$~is even. Then $
{\lambda}/{\ell}=\ell_{\pi_\circ}$~is even. Denote by~$\tilde{\mathbf{O}}^\zstroke(v')$ the set of orbits --\,under~$\sigma$\,-- of ordered sets of the form $(\{v'\}$,$S_1'$,$S_2'$,$\cdots$,$S_{n_2}')$, where $v'\in \sfV (\gamma^\zstroke)$ and $\pi' \mathrel{{:}{=}} 
(S_1$,$S_2$,$\cdots$,$S_{n_2}) \in \Pi_{n_2}(v')$. 
Let us rewrite $\sum_{v\in O} B_{\gamma^2}(v)$ again, this time by summing over the disjoint orbits of the pairs~$(v_0,\pi)$ under~$\sigma$, that is, summing over the elements in~$\tilde{\mathbf{O}}^\zstroke(v_\circ)$. 
(These orbits can be \emph{longer} than the orbit~$O_{v_\circ}$ of~$v_\circ$ since the orbit of an ordered partition~$\pi$ can be longer.) 
By the definition of a symmetry and by the properties of orbits of ordered sets (see 
Remark~\ref{Act}) we have that 
\begin{multline*}
\{(\sigma^i(v_\circ),\sigma^i(\pi)) : \pi\in\Pi_{n_2}(v_\circ), i=0,1,\dots,\ell-1\}\\
 = \{(\sigma^i(v_\circ),\sigma^i(\pi)) : O_{(v_\circ,\pi)}\in \tilde{\mathbf{O}}^\zstroke(v_\circ), i=1,2,\dots,\lambda(v_\circ,\pi)-1 \},
\end{multline*}
where $\lambda(v_\circ,\pi)$ is the length of the (long) orbit~$O_{(v_\circ,\pi)}$ of the ordered set~$(v_\circ,\pi)$.
Now 
re\-wri\-te equation~\eqref{TheSum} as
\[
\sum_{v\in O} B_{\gamma^2}(v) = \sum_{\pi\in\Pi_{n_2}(v_\circ)} \sum_{i=0}^{\ell-1} \gamma^\zstroke_{\sigma^i(v_\circ),\sigma^i(\pi)} = \sum_{O_{(v_\circ,\pi)}\in\tilde{\mathbf{O}}^\zstroke(v_\circ)} \sum_{i=0}^{\lambda(v_\circ,\pi)-1} \gamma^\zstroke_{\sigma^i(v_\circ),\sigma^i(\pi)}.
\]
In order to show that the above sum vanishes, it suffices to show that for every orbit $O_{(v_\circ,\pi)}\in\tilde{\mathbf{O}}^\zstroke(v_\circ)$ the sum 
\[
\sum_{i=0}^{\lambda(v_\circ,\pi)-1} \gamma^\zstroke_{\sigma^i(v_\circ),\sigma^i(\pi)}
\]
is equal to zero. Let $O_{(v_\circ,\pi_\circ)}\in\tilde{\mathbf{O}}^\zstroke(v_\circ)$. 
Again by the definition of a symmetry and by the properties of orbits of ordered sets we have
\begin{align*}
O_{(v_\circ,\pi_\circ)} &= \{(\sigma^i(v_\circ),\sigma^i(\pi)) : i=0,1,\dots,\lambda(v_\circ,\pi)-1 \}\\
 &= \bigsqcup_{0\leqslant j <\ell} \{(\sigma^{j+i\ell}(v_\circ),\sigma^{j+i\ell}(\pi_\circ)) : i=0,1,\dots,\frac{\lambda(v_\circ,\pi_\circ)}{\ell}-1 \}.
\end{align*}
Now it suffices to show that, for all $0\leqslant j <\ell$, the sum 
\[
\sum_{i=0}^{ -1+ 
{\lambda(v_\circ,\pi_\circ)}/{\ell} } \gamma^\zstroke_{\sigma^{j+i\ell}(v_\circ),\sigma^{j+i\ell}(\pi_\circ)}
\]
vanishes. Since $\ell$~is odd, by equation~\eqref{Odd} we have that $\gamma^\zstroke_{v_\circ,\pi_\circ} = -\gamma^\zstroke_{\sigma^\ell(v_\circ),\sigma^\ell(\pi_\circ)}$. Since $
{\lambda}/{\ell}$ is even, we can pair~$\gamma^\zstroke_{v_\circ,\pi_\circ}$ with~$\gamma^\zstroke_{\sigma^\ell(v_\circ),\sigma^\ell(\pi_\circ)}$ for $i=0$,\ $2$,\ $4$,\ $\dots$,\ $
({\lambda}/{\ell})-2$. All pairs are disjoint and cancel, so $\sum_{v\in O} B_{\gamma^2}(v)=0$ and the subtrahend vanishes.

\textbf{Part~2.} What 
is left to prove is that the minuend in equation~\eqref{Bracket}, 
\[
\sum_{v\in \sfV (\gamma^2)} B_{\gamma^\zstroke}(v) = \sum_{v\in \sfV (\gamma^2)} \sum_{\pi\in\Pi_{n_\zstroke}(v)} \gamma^2_{v,\pi},
\]
is equal to zero in~$\Gra$. 
It suffices to show that the sum $\sum_{\pi\in\Pi_{n_\zstroke}(v)} \gamma^2_{v_\circ,\pi}$ is zero in~$\Gra$ for any fixed vertex~$v_\circ\in\sfV(\gamma)$.
Denote by $\mathbf{O}^{\Pi_{n_\zstroke}(v_\circ)}$ the set of all orbits of ordered partitions in~$\Pi_{n_\zstroke}(v_\circ)$. 
Then
\[
\sum_{\pi\in\Pi_{n_\zstroke}(v)} \gamma^2_{v_\circ,\pi} = \sum_{O\in \mathbf{O}^{\Pi_{n_\zstroke}(v_\circ)}} \sum_{\pi\in O} \gamma^2_{v_\circ,\pi}.
\]
Hence it suffices to show that $\sum_{\pi\in O} \gamma^2_{v_\circ,\pi}$ is equal to zero in~$\Gra$ for any fixed orbit~$O$ in~$\mathbf{O}^{\Pi_{n_\zstroke}(v_\circ)}$.

Take~$O_\circ$ in~$\mathbf{O}^{\Pi_{n_\zstroke}(v_\circ)}$ and let~$\pi_\circ\in O_\circ$.
The orbit~$O_\circ$ can be seen as the orbit~$O_{\pi_\circ}$ of~$\pi_\circ$, and
\begin{equation}\label{Ell}
\sum_{\pi\in O} \gamma^2_{v_\circ,\pi} = \sum_{i=0}^{\ell_{\pi_\circ} - 1} \gamma^2_{v_\circ,\sigma^i(\pi_\circ)},
\end{equation}
where $\ell_{\pi_\circ}$ is the length of the orbit~$O_{\pi_\circ}$.

Assume $\ell_{\pi_\circ}$~is even. Consider $\gamma^2_{v_\circ,\pi_\circ} 
\mathrel{{=}{:}} 
\gamma^a$ and $\gamma^2_{v_\circ,\sigma(\pi_\circ)} \mathrel{{=}{:}} 
\gamma^b$. The graphs~$\gamma^a$ and~$\gamma^b$ are topologically equal under the 
matching of edges: $I^{(\zstroke)(b)} = \sigma(I^{(\zstroke)(a)})$, $II^{(\zstroke)(b)} = \sigma(II^{(\zstroke)(a)})$, $\ldots$, $K^{(\zstroke)(b)} = \sigma(K^{(\zstroke)(a)})$ and $\#^{(2)(b)} = \#^{(2)(a)}$ for all other respective edges in~$\gamma^{b}$ and~$\gamma^{a}$. The edge ordering of~$\gamma^b$ satisfies the equalities
\begin{align*}
\sfE(\gamma^b) &= I^{(\zstroke)(b)}\wedge II^{(\zstroke)(b)} \wedge \dots \wedge K^{(\zstroke)(b)} \wedge I^{(2)(b)} \wedge II^{(2)(b)} \wedge \dots \wedge K^{(2)(b)}\\
&= \sigma(I^{(\zstroke)})^{(a)}\wedge \sigma(II^{(\zstroke)})^{(a)} \wedge \dots \wedge \sigma(K^{(\zstroke)})^{(a)} \wedge I^{(2)(a)} \wedge II^{(2)(a)} \wedge \dots \wedge K^{(2)(a)}\\
&= - I^{(\zstroke)(a)}\wedge II^{(\zstroke)(a)} \wedge \dots \wedge K^{(\zstroke)(a)} \wedge I^{(2)(a)} \wedge II^{(2)(a)} \wedge \dots \wedge K^{(2)(a)}\\
&= - \sfE(\gamma^a).
\end{align*}
Hence $\gamma^2_{v_\circ,\pi_\circ} = \gamma^a = -\gamma^b = - \gamma^2_{v_\circ,\sigma(\pi_\circ)}$. Since the elements in~$O$ can be written as~$\sigma^i(\pi_\circ)$ for $0\leqslant i<\ell_{\pi_\circ}$, we have that 
\begin{equation}\label{Cancel}
\gamma^2_{v_\circ,\sigma^i(\pi_\circ)} = - \gamma^2_{v_\circ,\sigma^{i+1}(\pi_\circ)}
\end{equation}
for all $i=0$,$\ldots$,$\ell_{\pi_\circ}-2$. From the assumption that $\ell_{\pi_\circ}$~is even it follows that one can make a disjoint paring of all graphs in the right\/-\/hand side of equation~\eqref{Ell} such that all pairs cancel by equation~\eqref{Cancel}.

Assume that $\ell_{\pi_\circ}$~is odd. 
Then the graph~$\gamma^2_{v_\circ,\pi_\circ} \mathrel{{=}{:}} 
\gamma^a$ is topologically equal to itself via the following matching of edges:
$I^{(\zstroke)(a)} = \sigma^{\ell_{\pi_\circ}}(I^{(\zstroke)(b)})$, $II^{(\zstroke)(a)} = \sigma^{\ell_{\pi_\circ}}(II^{(\zstroke)(b)})$, $\ldots$, $K^{(\zstroke)(a)} = \sigma^{\ell_{\pi_\circ}}(K^{(\zstroke)(b)})$ and $\#^{(2)(a)} = \#^{(2)(b)}$ for all other respective edges in~$\gamma^{a}$ and $\gamma^{b} \mathrel{{:}{=}} 
\gamma^2_{v_\circ,\sigma^{\ell_{\pi_\circ}}}(\pi_\circ)$. The permutation $\sigma^{\ell_{\pi_\circ}}$ indeed satisfies $\sigma^{\ell_{\pi_\circ}}(\bar{N}(v)) = \bar{N}(\sigma^{\ell_{\pi_\circ}}(v))$ since the newly attached vertices stay 
invariant, therefore it defines a symmetry. The edge ordering of~$\gamma^b$ satisfies the chain of equalities
\begin{align*}
\sfE(\gamma^a) &= I^{(\zstroke)(a)}\wedge II^{(\zstroke)(a)} \wedge \dots \wedge K^{(\zstroke)(a)} \wedge I^{(2)(a)} \wedge II^{(2)(a)} \wedge \dots \wedge K^{(2)(a)}\\
&= \sigma^{\ell_{\pi_\circ}}(I^{(\zstroke)})^{(b)}\wedge \sigma^{\ell_{\pi_\circ}}(II^{(\zstroke)})^{(b)} \wedge \dots \wedge \sigma^{\ell_{\pi_\circ}}(K^{(\zstroke)})^{(b)} \wedge I^{(2)(b)} \wedge II^{(2)(b)} \wedge \dots \wedge K^{(2)(b)}\\
&= - I^{(\zstroke)(b)}\wedge II^{(\zstroke)(b)} \wedge \dots \wedge K^{(\zstroke)(b)} \wedge I^{(2)(b)} \wedge II^{(2)(b)} \wedge \dots \wedge K^{(2)(b)}\\
&= - \sfE(\gamma^b),
\end{align*}
where the minus sign follows from the assumption that $\ell_{\pi_\circ}$~is odd. Then we have about the graphs that $\gamma^2_{v_\circ,\pi_\circ} = \gamma^a = -\gamma^b = -\gamma^2_{v_\circ,\sigma^{\ell_{\pi_\circ}}(\pi_\circ)}$,
hence $\gamma^2_{v_\circ,\pi}$~is a zero graph. Since $v_\circ$~is chosen arbitrarily, all graphs in the sum~$\sum_{\pi\in O} \gamma^2_{v_\circ,\pi}$ are zero graphs, so this sum is equal to zero.

It follows that the minuend vanishes (since it is equal to a sum of zero graphs and/or graphs with a vanishing coefficient) and the statement follows.
\end{proof}

\begin{cor}\label{CordZero}
For any zero graph~$\gamma^\zstroke$, its differential $\Id(\gamma^\zstroke)=[\edge,\gamma^\zstroke]$ vanishes in~$\Gra$.
\end{cor}

This proposition is a special case of Theorem~\ref{TheoBracketZero}: namely, where $\gamma^2 = \pm \edge$ and the set of ordered partitions~$\Pi$ is restricted to the subset~$\bar{\Pi}$ of ordered partitions without empty components. (We recall that by Lemma~\ref{LeavesCancel}, the leaved graphs cancel out.)
The fact that $\Id(\gamma^\zstroke)=\boldsymbol{0}\in\Gra$, i.e.\ that after all the cancellations (see the above proof of Theorem~\ref{TheoBracketZero})
the differential of a \emph{zero} graph~$\gamma^\zstroke$ is a sum of zero graphs, can also be seen directly as follows. 

\begin{proof}[Sketch of the proof]
Let $\sigma$~be a symmetry of a zero graph~$\gamma^\zstroke$ such that $\gamma^\zstroke+\sigma(\gamma^\zstroke)=0$ in the vector space of graphs with ordered edge sets. Then
$0=\Id(0)=\Id\bigl(\gamma^\zstroke+\sigma(\gamma^\zstroke)\bigr) = \Id(\gamma^\zstroke)+
\Id\bigl(\sigma(\gamma^\zstroke)\bigr)= \Id(\gamma^\zstroke)+\sigma'\bigl(\Id(\gamma^\zstroke)\bigr)$, where the permutations~$\sigma'$ act on the old edges from~$\gamma^\zstroke$ and map the newly inserted edges~$\edge$ to the new edges (brought in by the vertex blow\/-\/ups in~$\Id$). Because the left-{} and right\/-\/hand sides of the equality $\Id(\gamma^\zstroke)=-\sigma'\bigl(\Id(\gamma^\zstroke)\bigr)$ contains the sums of topologically isomorphic graphs (one in the l.\/-\/h.s.\ versus one in the r.\/-\/h.s.), the sum of graphs~$\Id(\gamma^\zstroke)$ is equal to minus itself, whence the assertion.
\end{proof}

\section{The graded Jacobi identity}

\begin{lemma}
The bracket $[\cdot,\cdot]$ is graded skew\/-\/symmetric, i.e.\ for any graphs $\gamma^1$,\ $\gamma^2\in\Gra$ on $k_1$ and $k_2$ vertices, respectively, we have that 
$[\gamma^2,\gamma^1] = - (-)^{k_1 k_2}[\gamma^1,\gamma^2]$.
(This is clear from~\eqref{EqBracket}.)
\end{lemma}

\begin{theor}\label{ThJacobi}
The bracket $[\cdot,\cdot]$ on the space~$\Gra$ satisfies the 
graded Jacobi identity\textup{:}
\begin{equation}\label{EqGradJac}
\Jac(\gamma^1,\gamma^2,\gamma^3) \mathrel{{:}{=}} 
[[\gamma^1,\gamma^2],\gamma^3] 
+ (-)^{k_1 k_2+k_1 k_3} [[\gamma^2,\gamma^3],\gamma^1] 
+ (-)^{k_1 k_3+k_2 k_3} [[\gamma^3,\gamma^1],\gamma^2] = 0,
\end{equation}
for any graphs $\gamma^1$,\ $\gamma^2$,\ $\gamma^3$ on~$k_1$,\ $k_2$, and~$k_3$ vertices, respectively.

Moreover, the left\/-\/hand side of 
identity~\eqref{EqGradJac} 
equals zero viewed as an element of the vector space of formal sums of graphs 
\textup{(}i.e.\ the space \emph{without} relation~\eqref{EqRelation}\textup{).}
\end{theor}

\begin{proof}
The proof consists of two parts. 
We show first that all the graphs obtained in~\eqref{EqGradJac} by applying a repeated 
insertion of any graphs into~$\gamma^3$ cancel out. Secondly, we prove that the left\/-\/hand side of~\eqref{EqGradJac} is graded skew symmetric w.r.t.\ 
a permutations $\sigma\in S_3$ of the arguments~$\gamma^1$, $\gamma^2$,~$\gamma^3$.

\textbf{Part 1.} We have that
\begin{multline*}
[[\gamma^1,\gamma^2],\gamma^3] 
  + (-)^{k_1k_2+k_1k_3} [[\gamma^2,\gamma^3],\gamma^1] 
  + (-)^{k_1k_3+k_2k_3} [[\gamma^3,\gamma^1],\gamma^2] \\
= [[\gamma^1,\gamma^2],\gamma^3] 
   - (-)^{k_1k_2+k_1k_3+k_1(k_2+k_3)} [\gamma^1,[\gamma^2,\gamma^3]] 
   + (-)^{k_1k_3+k_2k_3+k_1k_3+k_2(k_1+k_3)} [\gamma^2,[\gamma^1,\gamma^3]] \\
= [[\gamma^1,\gamma^2],\gamma^3] 
   - [\gamma^1,[\gamma^2,\gamma^3]] 
   + (-)^{2(k_1k_3+k_2k_3)+k_1k_2} [\gamma^2,[\gamma^1,\gamma^3]].
\end{multline*}
If we restrict ourselves to the graphs obtained by applying twice an 
insertion of the 
graphs into~$\gamma^3$, we obtain
\begin{align*}
&\underbrace{(\gamma^1\circ_i\gamma^2)\circ_i\gamma^3}_{a_1} - \underbrace{(\gamma^2\circ_i\gamma^1)\circ_i\gamma^3}_{b_1} - \underbrace{\gamma^1\circ_i(\gamma^2\circ_i\gamma^3)}_{a_2}
- \underbrace{\left(^{\gamma^1\circ_i} _{\gamma^2\circ_i} \gamma^3\right) }_{c_1}\\
&{}\quad{}+ \underbrace{(-)^{2(k_1k_2+k_1k_3+k_2k_3)} (\gamma^2\circ_i\gamma^1)\circ_i\gamma^3}_{b_2} + \underbrace{(-)^{2(k_1k_2+k_1k_3+k_2k_3)} \left(^{\gamma^1\circ_i} _{\gamma^2\circ_i} \gamma^3\right) }_{c_2},
\end{align*}
where each graph is multiplied by the factor that corresponds to the change of edge ordering into $\sfE(\gamma^1)\wedge\sfE(\gamma^2)\wedge\sfE(\gamma^3)$. Here $\left(^{\gamma^1\circ_i} _{\gamma^2\circ_i} \gamma^3\right)$ equals the sum of those graphs present in $\gamma^1\circ_i(\gamma^2\circ_i\gamma^3)$ where $\gamma^1$ is inserted into a vertex of $\gamma^2\circ_i\gamma^3$ that initially belonged to $\gamma^3$ (and not to $\gamma^2$). 
Note that in part $c_2$ we 
replaced $\left(^{\gamma^2\circ_i} _{\gamma^1\circ_i} \gamma^3\right)$ by $(-)^{k_1k_2}\left(^{\gamma^1\circ_i} _{\gamma^2\circ_i} \gamma^3\right)$, where the graded sign appears due to the edge ordering. We see that $a_1$~cancels with~$a_2$, $b_1$~cancels with~$b_2$, and $c_1$~cancels with~$c_2$, so the formal sum of these graphs equals zero.

\textbf{Part 2.} It suffices to verify the skew symmetry of the left\/-\/hand side in the
graded identity~\eqref{EqGradJac} for any two generators of the group~$S_3$. 
We choose the two\/-\/cycles $(12)$ and~$(23)$ as its generators.
For the transposition~$\sigma=(12)$ we have that
\begin{multline*}
\Jac(\gamma^2,\gamma^1,\gamma^3) = [[\gamma^2,\gamma^1],\gamma^3] + (-)^{k_2k_1+k_2k_3} [[\gamma^1,\gamma^3],\gamma^2] + (-)^{k_2k_3+k_1k_3} [[\gamma^3,\gamma^2],\gamma^1]\\
= -(-)^{k_1k_2}[[\gamma^1,\gamma^2],\gamma^3] -(-)^{k_1k_2+k_1k_3+k_2k_3}[[\gamma^3,\gamma^1],\gamma^2] - (-)^{2k_2k_3+k_1k_3}[[\gamma^2,\gamma^3],\gamma^1]\\
=  - (-)^{k_1k_2}\left([[\gamma^1,\gamma^2],\gamma^3] 
   + (-)^{k_1k_2+k_1k_3} [[\gamma^2,\gamma^3],\gamma^1] 
   + (-)^{k_1k_3+k_2k_3} [[\gamma^3,\gamma^1],\gamma^2]\right)\\
= -(-)^{k_1k_2} \Jac(\gamma^1,\gamma^2,\gamma^3).
\end{multline*}
We see that $\epsilon(\sigma)=(-)^\sigma\cdot(-)^{\#\sfE(\gamma^1)\cdot\#\sfE(\gamma^2)}$
is the Koszul sign of the transposition~$\sigma=(12)$ of two homogeneous graded objects~$\gamma^1$ and~$\gamma^2$ (of gradings~$k_1$ and~$k_2$, respectively).
Likewise, for $\sigma=(23)$ we have that
\begin{multline*}
\Jac(\gamma^1,\gamma^3,\gamma^2) = [[\gamma^1,\gamma^3],\gamma^2] + (-)^{k_1k_3+k_1k_2} [[\gamma^3,\gamma^2],\gamma^1] + (-)^{k_1k_2+k_2k_3} [[\gamma^2,\gamma^1],\gamma^3]\\
= -(-)^{k_1k_3}[[\gamma^3,\gamma^1],\gamma^2] -(-)^{k_1k_2+k_1k_3+k_2k_3}[[\gamma^2,\gamma^3],\gamma^1] - (-)^{2k_1k_2+k_2k_3}[[\gamma^1,\gamma^2],\gamma^3]\\
=  - (-)^{k_2k_3}\left([[\gamma^1,\gamma^2],\gamma^3] 
   + (-)^{k_1k_2+k_1k_3} [[\gamma^2,\gamma^3],\gamma^1] 
   + (-)^{k_1k_3+k_2k_3} [[\gamma^3,\gamma^1],\gamma^2]\right)\\
=  -(-)^{k_2k_3} \Jac(\gamma^1,\gamma^2,\gamma^3).
\end{multline*}
In conclusion, $\Jac\bigl(\gamma^{\sigma(1)},\gamma^{\sigma(2)},\gamma^{\sigma(3)}\bigr)
=\epsilon(\sigma)\cdot\Jac(\gamma^1,\gamma^2,\gamma^3)$ with respect to the Koszul sign~$\epsilon(\sigma)$ of any permutation~$\sigma\in S_3$ of the graded objects.
This completes the proof.
\end{proof}

\begin{cor}[$\Id^2$ vanishes]\label{EasyddZeroInGra}
Let $\gamma$~be a single graph on~$k$ edges. Then~$\Id^2(\gamma)=\boldsymbol{0}\in\Gra$.
\end{cor}

\begin{proof}
Indeed, let $\gamma^1\mathrel{{:}{=}}\edge$ and $\gamma^2\mathrel{{:}{=}}\edge$ in the Jacobi identity~\eqref{EqGradJac}, whence
\begin{multline*}
[[\edge,\edge],\gamma] 
  + (-)^{1\cdot 1+1\cdot k} [[\edge,\gamma],\edge]
  + (-)^{1\cdot k+1\cdot k} [[\gamma,\edge],\edge] \\
{}= [[\edge,\edge],\gamma] 
  - (-)^{2(k+1)}[\edge,[\edge,\gamma]]
  + (-)^{2k+1}[\edge,[\edge,\gamma]] =0,
\end{multline*}
so that $\Id^2(\gamma)=\tfrac{1}{2}[[\edge,\edge],\gamma]$.
Because $[\edge,\edge]$ is a zero graph (see Example~\ref{ExLeavesCancel} on p.~\pageref{ExLeavesCancel}), the claim is proved.
\end{proof}

\begin{cor}
Let $\gamma^1$ and~$\gamma^2$ be --\,without loss of generality, homogeneous with respect to the number of edges\,-- $\Id$-\/cocycles. Then their Lie bracket~$[\gamma^1,\gamma^2]$ is a $\Id$-\/cocycle as well.
\end{cor}

\begin{proof}
We have that $\pm\Id\bigl([\gamma^1,\gamma^2]\bigr) =
\pm[\Id(\gamma^1),\gamma^2] \pm [\Id(\gamma^2),\gamma^1] = \boldsymbol{0}\in\Gra$ by Theorem~\ref{TheoBracketZero}. One proceeds now by linearity over the grading\/-\/homogeneous components in each cocycle~$\gamma^1,\gamma^2\in\ker\Id$.
\end{proof}

{\small
\subsubsection*{Acknowledgements}
The authors thank R.\,Buring and S.\,Taams for helpful discussions.
The authors are grateful to the organizers of 32nd International colloquium on group theoretical methods in Physics (9--13~July 2018, CVUT~Prague, Czech Republic) for a warm atmosphere during the meeting and partial financial support of~NJR.
The research of AVK was partially supported by RUG project~106552.

}

\end{document}